\documentclass[a4paper]{amsart}

\pdfoutput=1

\usepackage[utf8]{inputenc}
\usepackage[T1]{fontenc}
\usepackage{lmodern}
\usepackage{amsthm, amssymb, amsmath, amsfonts, mathrsfs}
\usepackage[colorlinks=true, pdfstartview=FitV, linkcolor=blue, citecolor=blue, urlcolor=blue,pagebackref=false]{hyperref}

\usepackage{microtype}

\newtheorem{thm}{Theorem}[section]
\newtheorem{prop}[thm]{Proposition}
\newtheorem{lem}[thm]{Lemma}
\newtheorem{cor}[thm]{Corollary}
\theoremstyle{remark}
\newtheorem{rem}[thm]{Remark}
\theoremstyle{definition}
\newtheorem{defi}[thm]{Definition}

\renewcommand{\le}{\leqslant}
\renewcommand{\ge}{\geqslant}
\renewcommand{\subset}{\subseteq}
\newcommand{\mcl}{\mathcal}

\newcommand{\msc}{\mathscr}
\newcommand{\B}{\mathbb{B}}
\newcommand{\e}{\mathbf{e}}
\newcommand{\E}{\mathbb{E}}

\newcommand{\N}{\mathbb{N}}
\newcommand{\na}{\nabla}
\newcommand{\Ll}{\left}
\newcommand{\Rr}{\right}
\newcommand{\lhs}{left-hand side}

\newcommand{\1}{\mathbf{1}}
\newcommand{\R}{\mathbb{R}}

\newcommand{\Z}{\mathbb{Z}}

\renewcommand{\P}{\mathbb{P}}

\newcommand{\ov}{\overline}
\newcommand{\un}{\underline}

\newcommand{\eps}{\varepsilon}
\renewcommand{\d}{{\mathrm{d}}}

\renewcommand{\o}{\omega}

\newcommand{\M}{\mcl{M}_q}
\newcommand{\MM}{\mathscr{M}_q}

\newcommand{\Ma}{\mathsf{M}}
\newcommand{\mi}{\mathsf{m}\,}
\renewcommand{\t}{^{(t)}}

\numberwithin{equation}{section}

\title[Anchored Nash inequalities and heat kernel bounds]{Anchored Nash inequalities and heat kernel bounds for static and dynamic degenerate environments}
\author{Jean-Christophe Mourrat, Felix Otto}

\address[Jean-Christophe Mourrat]{Ecole normale sup\'erieure de Lyon, CNRS, Lyon, France}

\address[Felix Otto]{Max Planck Institute for Mathematics in the Sciences, Leipzig, Germany}

\begin{document}

\begin{abstract}

We introduce anchored versions of the Nash inequality. They allow to control the $L^2$ norm of a function by Dirichlet forms that are not uniformly elliptic. We then use them to provide heat kernel upper bounds for diffusions in degenerate static and dynamic random environments. As an example, we apply our results to the case of a random walk with degenerate jump rates that depend on an underlying exclusion process at equilibrium.

\bigskip

\noindent \textsc{MSC 2010:} 26D10, 35K65, 35B65, 60K37.

\medskip

\noindent \textsc{Keywords:} Nash inequality, heat kernel, diffusion in dynamic random medium.

\end{abstract}
\maketitle
%
%
%
%
%
%
%
%
\section{Introduction}
\label{s:intro}

In the study of reversible random walks in degenerate static random environments, heat kernel estimates have played a key role. Consider a random walk evolving on a super-critical (infinite) percolation cluster. As was pointed out in \cite{sidszn}, the problem of showing that the walk satisfies a central limit theorem for almost every realisation of the environment (i.e.\ a ``quenched'' CLT) is reduced to showing a spatially averaged (instead of uniform) sublinear control of the corrector if the heat kernel is known to satisfy a diffusive upper bound. For the case of percolation, such a bound was obtained in \cite{matrem,bar}. This bound indeed became a primary ingredient in the subsequent proofs that the random walk on percolation clusters and other degenerate i.i.d.\ environments satisfies a quenched central limit theorem \cite{sidszn,berbis,matpia,bispre,mathieu, abdh}. (A notable exception is a geometric argument introduced in \cite{berbis}, which in two dimensions does not use any a priori heat kernel bounds.) We refer to \cite{biskup-survey} (in particular, Section~4.3 therein) and \cite{kum} for reviews. Related ideas of regularity theory lead to a proof of the local central limit theorem in~\cite{barham}.

The approach to heat kernel bounds in \cite{matrem} is based on studying isoperimetric inequalities on boxes, with an effective dimension that is slowly tuned up to $d$ as the size of the box increases. In \cite{bar}, upper and lower bounds are proved (including optimal off-diagonal control) using a fine decomposition of the space into good and bad boxes, where a box is said to be good if the Poincaré inequality (with the standard scaling) holds in that box. The case of i.i.d.\ conductances with a polynomial tail near zero is investigated in \cite{bkm} using a combination of probabilistic arguments and Harnack inequalities.

One can also consider environments that do not possess strong decorrelation properties at large distances. Assuming a certain moment condition (from above and below) on the conductances, it was shown in \cite{ger1} that the random walk satisfies a quenched central limit theorem. The proof relies on weighted Sobolev inequalities, and on Moser's iteration scheme. Harnack inequalities and a local CLT were then established in \cite{ger2}, and Gaussian heat kernel bounds in~\cite{ger3}. Related ideas were developped in \cite{fankom1} to prove a quenched invariance principle for Brownian motion subject to an incompressible drift. In a different direction, diffusions on clusters of correlated percolation models such as random interlacements are studied in \cite{prs} via isoperimetric inequalities. 


Some assumptions on the law of the conductivities are necessary for diffusive heat kernel bounds to hold. Indeed, even in the setting of i.i.d.\ conductances, where the quenched CLT is known to hold in full generality \cite{bispre,mathieu,abdh}, diffusive heat kernel upper bounds will not hold if the conductivities are allowed to take very small values~\cite{bbhk}. The basic mechanism is that a portion of space surrounded by a region of low conductivity acts as a trap for the random walk. This phenomenon was analysed with precision in \cite{boukha1,bisbou,blrv}.


\medskip

The aim of this work is to introduce a new method to prove diffusive heat kernel bounds for degenerate environments. Our first main result, obtained in Section~\ref{s:nash}, is an anchored version of the Nash inequality. The standard Nash inequality gives a control of the $L^2$ norm of a function $f$ in terms of $\|\na f\|_2$ and $\|f\|_1$. We show that one can control the $L^2$ norm of $f$ in terms of $\|w \na f\|_2$, $\|f\|_1$ and $\||x|^{p/2} f\|_2$, where $w$ is a weight function that can for instance be taken to be the square root of the conductance field (in particular, it is not bounded away from $0$ in cases of interest). Note that the translation invariance of the standard Nash inequality is broken by the term $\||x|^{p/2} f\|_2$.

We then show in Section~\ref{s:static} that despite this extra term $\||x|^{p/2} f\|_2$, one can recover diffusive heat kernel upper bounds for static degenerate environments satisfying a moment condition. 

Our main motivation is to derive heat kernel bounds for \emph{dynamic} degenerate random environments. We are not aware of previous results in this context, although quenched central limit theorems have been proved in some cases (see in particular~\cite{fankom2}). The dynamics of the environment can facilitate anomalous behaviour of the heat kernel~\cite{buckley}. As a simple example, one can construct a dynamic environment such that the probability of return to the origin decays from $1$ to almost zero, and then climb back up to almost $1/4$. Indeed, take $z$ a neighbour of the origin, and consider that initially, the edge connecting the origin to $z$ is open, and these two sites are disconnected from the rest of the graph. The probability for the walk to be on $z$ goes to $1/2$ as time passes. Then change the environment so that the site $z$ is isolated from the rest, and otherwise every edge is open. The probability for the walk to be at the origin decays to $0$. Then move back to the initial situation with the origin and $z$ connected, and disconnected from the rest: the probability for the walk to be at the origin moves back to being almost $1/4$. 

Let us denote by $p_{s,t}(x,y)$ the probability for the walk started at $x$ at time $s$ to be at $y$ at time $t\ge s$. It is more convenient to aim first for bounds on 
$$
\mcl E_t := \sum_{x \in \Z^d} (p_{0,t}(0,x))^2,
$$
in particular because, as opposed to $p_{0,t}(0,0)$, this quantity is monotone in $t$. 
In Section~\ref{s:dynamic}, we provide a general criterion that allows to prove diffusive bounds of the form
\begin{equation}
\label{e:bound-E}
\mcl E_t \le \frac{\mcl X}{t^{d/2}},
\end{equation}
where $\mcl X$ is a random variable. In order to illustrate the relevance of our criterion, we study in Section~\ref{s:example} an example of a random walk whose (symmetric) jump rates $(a_t(e))$ are a local function of an underlying simple exclusion process at equilibrium. We show that in this case, \eqref{e:bound-E} holds with a random variable $\mcl X$ that has finite moments of every order, regardless of the density of particles in the exclusion process, or of the specific definition of the jump rates. (In particular, it may be that at any given time, the set of sites that are reachable by the random walk is finite.) We also show that for every $r \ge 1$, $\sup_t \E[(t^{d/2} p_{0,t}(0,0))^r]$ is finite, and that for every $\eps > 0$, $\E[(\sup_t t^{d/2 -\eps} p_{0,t}(0,0))^r]$ is also finite. In order for these results to hold, we need only little information on the underlying dynamics. Roughly speaking, we only use the fact that the jump rates are very unlikely to vanish over an extended period of time. In particular, we do not use the fact that the simple exclusion process has a polynomial mixing rate of order $t^{-d/2}$ (as was shown in \cite{berzeg,jlqy}).

\subsection{Notation}

We let $|\cdot|_* = |\cdot| \vee 1$. For any positive integer $r$, we denote by $B_r$ the box $\{-r,\ldots, r\}^d$, and $\B_r$ the set of edges with both end-points in $B_r$. For any $p \in [1,\infty]$, we denote by
$\| \cdot \|_{L^p_r}$ the norm of the space $L^p(B_r)$ or $L^p(\B_r)$ (which of the two will be clear from the context), i.e.\
$$
\| f \|_{L^p_r} = \Ll(\sum_{x \in B_r} |f(x)|^p\Rr)^{1/p} \quad \text{ or } \quad \| f \|_{L^p_r} = \Ll(\sum_{e \in \B_r} |f(e)|^p\Rr)^{1/p},
$$
with the usual interpretation as a supremum if $p = \infty$. We also write
$$
\| f \|_{p} = \Ll(\sum_{x \in \Z^d} |f(x)|^p\Rr)^{1/p} \quad \text{ or } \quad \| f \|_{p} = \Ll(\sum_{e \in \B} |f(e)|^p\Rr)^{1/p}.
$$

%
%
%
%
%
%
%
%
\section{Anchored Nash inequality}
\label{s:nash}
The aim of this section is to prove the following theorem.

\begin{thm}[Anchored Nash inequality]
\label{t:nash}
Let $p \in (d,+\infty)$, $q \in (d,+\infty]$, and $\theta \in [\theta_c,1]$, where $\theta_c \in [0,1)$ is defined by
\begin{equation}
\label{e:def:theta}
\frac{1}{\theta_c} = 1 + \frac{dp+2p}{dp+2d}\Ll(\frac{q}{d} - 1\Rr).
\end{equation}
Define $\alpha, \beta, \gamma \in [0,1)$ by
\begin{equation}
\label{e:def:alpha}
\alpha = (1-\theta) \frac{d}{d+2} + \theta \frac{p}{p+2}, \qquad \beta = (1-\theta) \frac{2}{d+2}, \qquad \gamma = \theta \frac{2}{p+2}.
\end{equation}
There exists $C < \infty$ such that for any functions $f : \Z^d \to \R$ and $w : \B \to \R_+$,
\begin{equation}
\label{e:nash}
\|f\|_{2} \le C \Ll(\mcl{M}_q \,  \|w \, \na f\|_{2}\Rr)^\alpha \ \|f\|_{1}^\beta \ \| |x|_*^{p/2} f \|_{2}^\gamma,
\end{equation}
where
\begin{equation}
\label{e:def:Mq}
\mcl{M}_q^q  = \mcl{M}_q^q(w) = \sup_{r \in \N} \frac{1}{|\B_r|} \sum_{e \in \B_r} w^{-q}(e).
\end{equation}
\end{thm}
\begin{rem}
By setting $q = + \infty$ and $\theta = 0$, we recover the usual Nash inequality, provided that $\mcl{M}_\infty$ is finite (i.e.\ that $w^{-1}$ is bounded). For $\theta > 0$ (and thus $\gamma > 0$), the inequality is no longer translation invariant. Indeed, a crucial aspect of the inequality is that in the definition of $\mcl{M}_q$, the supremum is taken only over boxes that are centred around the origin. The price one has to pay for this is the appearance of the extra term $\| |x|_*^{p/2} u \|_{2}^\gamma$, which is of course not translation invariant either.
\end{rem}
\begin{rem}
Clearly, $\alpha + \beta + \gamma = 1$, as must be the case by homogeneity. 
\end{rem}
\begin{rem}
Our proof applies as well to the setting of continuous space. Considerations of scale invariance then show that the exponents $\alpha$, $\beta$ and $\gamma$ have to be of the form \eqref{e:def:alpha} for some~$\theta$.
\end{rem}
\begin{rem}
\label{r:value-of-theta}
As far as our applications are concerned, the precise value of $\theta_c$ is unimportant, the only crucial point being that $\theta_c < 1$.
\end{rem}

Before we turn to the proof of Theorem~\ref{t:nash}, we recall some classical functional inequalities for the reader's convenience. 

\begin{thm}[Poincaré-Sobolev inequalities]
\label{t:poisob}
Let $1 < p < d$ and let $p^\star$ be such that $d/p^\star = d/p-1$. There exists a constant $C < \infty$ such that for every $r$ and $f : B_r \to \R$,
\begin{equation}
\label{e:poisob}
\|f-\ov{f}_r\|_{L^{p^\star}_r} \le C \|\na f\|_{L^p_r},
\end{equation}
where we write
$$
\ov{f}_r = \frac{1}{|B_r|} \sum_{x \in B_r} f(x).
$$
\end{thm}

\begin{proof}
For every $x,y \in B_r$, we give ourselves a path $\gamma_{x,y}$ of nearest-neighbour points of $\Z^d$ from $x$ to $y$ that stays as close as possible to the segment joining $x$ and $y$. To be precise, by symmetry we may assume that $y \ge x$ (i.e.\ each coordinate of $y$ is larger than the corresponding coordinate of $x$). In this case, we construct a path $x_0 = x, x_1, \ldots, x_l = y$ as follows. We write $I_{x,y}$ for the line segment joining $x$ and $y$ in $\R^d$. Assuming that $x_0, \ldots, x_j$ are already built, we define
$$
i = \arg \min \{ \mathrm{dist} (x_j+\e_{i'},I_{x,y}), \ 1 \le i' \le d \},
$$
(take the smallest one in case of ties), and put $x_{j+1} = x_j + \e_i$. Such a path stays at distance at most $\sqrt{d}$ from $I_{x,y}$, and has no self-intersections. 
We have
$$
\Ll|f(y) - f(x)\Rr| \le \sum_{e \in \gamma_{x,y}} |\na f|(e),
$$
with obvious notation. Since
\begin{equation}
\label{unif-average}
f(x) - \ov{f}_r = \frac{1}{|B_r|} \sum_{y \in B_r} (f(x)-f(y)),
\end{equation}
the triangle inequality yields
$$
\Ll|f(x) - \ov{f}_r\Rr| \le \frac{1}{|B_r|} \sum_{\substack{ y \in B_r \\ e \in \gamma_{x,y}}} |\na f|(e).
$$
One can check that there exists a constant $C$ (independent of $r$) such that for every $x \in B_r$ and $e \in \B$,
$$
\Ll| \Ll\{ y \in B_r : e \in \gamma_{x,y}  \Rr\} \Rr| \le C r \Ll( \frac{r}{1+|x-\un{e}|} \Rr)^{d-1}.
$$
As a consequence,
\begin{equation}
\label{e:HLS}
\Ll|f(x) - \ov{f}_r\Rr| \lesssim \sum_{e \in \B_r} (1+ |x-\un{e}|)^{-(d-1)} |\na f|(e).
\end{equation}
The desired result now follows from the Hardy-Littlewood-Sobolev inequality, see e.g.\ \cite[Theorems~1.5 and 1.7]{BCD}.
\end{proof}

\begin{rem}
\label{r:gener-poisob}
Theorem~\ref{t:poisob} can be generalised in the following way. Under the same conditions, if $g : B_r \to \R$ is such that $\ov{g}_r = 1$, then
$$
\|f-\ov{(fg)}_r\|_{L^{p^\star}_r} \le C \|g\|_{L^\infty_r} \|\na f\|_{L^p_r}.
$$
Indeed, it suffices to 
observe that
$$
\Ll| \ov{(fg)}_r - \ov f_r \Rr| = \Ll| \frac{1}{|B_r|} \sum_{x \in B_r} g(x) \Ll(  f(x) - \ov f_r \Rr) \Rr| \le \|g\|_{L^\infty_r} \, \frac{1}{|B_r|^{{1/p^\star}}} \|f-\ov{f}_r\|_{L^{p^\star}_r},
$$
and to use the triangle inequality and Theorem~\ref{t:poisob} (with $\|g\|_{L^{\infty}_r} \ge 1$).
\end{rem}
In dimension $2$, we will also need to use inequality \eqref{e:poisob} in the critical case $p^\star = 2$, $p = 1$. The following theorem gives the result in any dimension.
\begin{thm}[Isoperimetric inequality]
\label{t:isop}
Let $p^\star = d/(d-1)$. There exists $C < \infty$ such that for every $r$ and $f : B_r \to \R$,
$$
\|f-\ov{f}_r\|_{L^{p^\star}_r} \le C \|\na f\|_{L^1_r}.
$$
\end{thm}
\begin{proof}
Let $m$ be a median of $f$. It suffices to show that
\begin{equation}
\label{e:isop-suff}
\|f-m\|_{L^{p^\star}_r} \le C \|\na f\|_{L^1_r}.
\end{equation}
Indeed, if $m > \ov f_r$, then by the defining property of the median,
$$
\frac{|B_r|}{2} \le |\{x \in B_r : f(x) \ge m\}| \le \frac{1}{(m-\ov f_r)^{p^\star}} \|f-\ov f_r\|_{L^{p^\star}_r}^{p^\star}.
$$
If $m < \ov f_r$, the symmetric argument shows that in every case, 
$$
|m-\ov f_r| \le \frac{2}{|B_r|^{1/{p^\star}}} \|f-\ov f_r\|_{L^{p^\star}_r}.
$$
By the triangle inequality, it thus suffices to show \eqref{e:isop-suff}. Without loss of generality, we may assume that $m = 0$. Let $g = f \vee 0$. Since $|\nabla g| \le |\nabla f|$, it suffices to show that
\begin{equation}
\label{e:isop-suf2}
\|g\|_{L^{p^\star}_r} \le C \|\na g\|_{L^1_r} \qquad (g \text{ s.t. } |\{x \in B_r : g(x) = 0\}| \ge |B_r|/2).
\end{equation}
For $t \ge 0$, let $\msc L_t = \{x \in B_r : g(x) > t\}$. Since $g = \int_0^\infty \1_{\msc L_t} \, \d t$, we have by Minkowski's inequality
$$
\|g\|_{L^{p^\star}_r} \le \int_0^\infty |\msc L_t|^{1/{p^\star}} \, \d t.
$$
The isoperimetric inequality \cite[Theorem~3.3.15]{sc} ensures that there exists $C < \infty$ such that for every $A \subset B_r$ with $|A| \le |B_r|/2$, we have $|A|^{(d-1)/d} \le C |\partial A|$, where
$$
\partial A = \{e = \{x,y\} \in \B_r : x \in A \text{ and } y \notin A\}.
$$
Hence,
\begin{align*}
\|g\|_{L^{p^\star}_r} & \le C \int_0^\infty |\partial \msc L_t| \, \d t \\
 & = C \int_0^\infty \sum_{x \sim y: g(x) < t \le g(y)} 1 \, \d t \\
 & = C \sum_{x \sim y : g(x) < g(y)} (g(y) - g(x)) = C \|\nabla g\|_{L^1_r},
\end{align*}
which concludes the proof.
\end{proof}
\begin{rem}
\label{r:gener2}
Under the conditions of Theorem~\ref{t:isop}, if $g : B_r \to \R$ is such that $\ov g_r = 1$, then 
$$
\|f-\ov{(fg)}_r\|_{L^{p^\star}_r} \le C \|g\|_{L^\infty_r} \|\na f\|_{L^1_r}.
$$
Indeed, the proof is identical to that of Remark~\ref{r:gener-poisob}.
\end{rem}


\begin{proof}[Proof of Theorem~\ref{t:nash}] The proof will be divided into two steps. We show the result for $\theta = 1$ in the first step, and for $\theta = \theta_c$ in the second step. The full result then follows by interpolation.

\medskip

\noindent \emph{Step 1.} We show that the theorem holds for $\theta = 1$. Let $p'$ be such that
\begin{equation}
\label{e:def:p}
\frac{1}{p'} = \frac1d + \frac12.
\end{equation}
Let $R$ be an even positive integer. We define a function $g : B_R \to \R_+$ that is $0$ in $B_{R/2}$, and is constant on the complementary ``annulus'' $A_R := B_R \setminus B_{R/2}$. The constant is specified by imposing $\ov{g}_R = 1$. We use the Poincaré-Sobolev inequality in the form given by Remark~\ref{r:gener-poisob} in dimension $d \ge 3$, and the isoperimetric inequality in the form given by Remark~\ref{r:gener2} when $d = 2$, to derive
$$
\|f - \ov{(fg)}_R\|_{L^2_{R}} \lesssim \|\na f\|_{L^{p'}_R}.
$$
In view of the definition of $g$, we have 
$$
\|\ov{(fg)}_R\|_{L^2_{R}} \lesssim |B_R|^{-1/2} \ \|f\|_{L^1(A_R)} \le \|f\|_{L^2(A_R)} \le R^{-p/2}\ \||x|_*^{p/2} f\|_{L^2_R},
$$
where we used Jensen's inequality for the second inequality. By Hölder's inequality (and recalling \eqref{e:def:p}), we also have
$$
\|\na f\|_{L^{p'}_R} \le \|w^{-1}\|_{L^d_R} \ \|w \na f\|_{L^2_R}.
$$
By Jensen's inequality (and since $q \ge d$), 
$$
\|w^{-1}\|_{L^d_R} \le |B_R|^{1/d-1/q} \ \|w^{-1}\|_{L^q_R} \le |B_R|^{1/d} \, \mcl{M}_q,
$$
with $\mcl{M}_q$ as in \eqref{e:def:Mq}. We thus get
$$
\|f\|_{L^2_R} \lesssim |B_R|^{1/d} \, \M \ \|w \na f\|_{L^2_R} + R^{-p/2}\ \||x|_*^{p/2} f\|_{L^2_R}.
$$
Clearly,
\begin{equation}
\label{e:step10}
\|f\|_{L^2(\Z^d \setminus B_R)} \le R^{-p/2} \||x|_*^{p/2} f\|_{L^2},
\end{equation}
so we established
\begin{equation}
\label{e:step11}
\|f\|_{L^2} \lesssim |B_R|^{1/d} \, \M \ \|w \na f\|_{L^2} + R^{-p/2}\ \||x|_*^{p/2} f\|_{L^2},
\end{equation}
uniformly over $R$ even positive integer. It suffices to change the implicit multiplicative constant in \eqref{e:step11} to ensure that the inequality
\begin{equation}
\label{e:step11.5}
\|f\|_{L^2} \lesssim R \, \M \ \|w \na f\|_{L^2} + R^{-p/2}\ \||x|_*^{p/2} f\|_{L^2}
\end{equation}
remains valid for every real $R \ge 1$. Since $\|f\|_{L^2} \le \||x|_*^{p/2} f\|_{L^2}$, the inequality extends to all $R > 0$. Optimizing over $R$, we obtain
$$
\|f\|_{L^2} \lesssim \Ll(\M \ \|w \na f\|_{L^2}\Rr)^{p/(p+2)} \ \||x|_*^{p/2} f\|_{L^2}^{2/(p+2)},
$$
as desired.

\medskip 

\noindent \emph{Step 2.} We now show that Theorem~\ref{t:nash} holds for $\theta = \theta_c$.

Let $Q$ be a box of size $r \in \N$, We use again $p'$ as in \eqref{e:def:p}.
By the Poincaré-Sobolev (or isoperimetric) inequality,
$$
\|f-\ov{f}_Q\|_{L^2(Q)} \lesssim \|\na f\|_{L^{p'}(Q)},
$$
where we denote by $\ov{f}_Q$ the average of $f$ over the box $Q$. 
By \eqref{e:def:p} and Hölder's inequality, we have
$$
\|\na f\|_{L^{p'}(Q)} \le \|w^{-1}\|_{L^d(Q)} \, \|w \na f \|_{L^2(Q)},
$$
while 
$$
\|\ov{f}_Q\|_{L^2(Q)} \le |Q|^{-1/2} \, \|f\|_{L^1(Q)}, 
$$
so that
$$
\|f\|_{L^2(Q)} \lesssim  \|w^{-1}\|_{L^d(Q)} \, \|w \na f \|_{L^2(Q)} + |Q|^{-1/2} \, \|f\|_{L^1(Q)},
$$
and thus
\begin{equation}
\label{e:step12}
\|f\|^2_{L^2(Q)} \lesssim  \|w^{-1}\|^2_{L^d(Q)} \, \|w \na f \|^2_{L^2(Q)} + |Q|^{-1} \, \|f\|^2_{L^1(Q)}.
\end{equation}
Let $R \in \N$, $R \ge r$. By Jensen's inequality,
$$
|Q|^{-1/d} \, \|w^{-1}\|_{L^d(Q)} \le |Q|^{-1/q} \, \|w^{-1}\|_{L^q(Q)}.
$$
For $Q \subset B_R$, $\|w^{-1}\|_{L^q(Q)} \le \|w^{-1}\|_{L^q_R} \le |B_R|^{1/q} \, \M$,  and we have
\begin{equation}
\label{e:step13}
\|w^{-1}\|_{L^d(Q)} \le |Q|^{1/d - 1/q} \, |B_R|^{1/q} \, \M.
\end{equation}
If $(Q'_i)$ is a finite collection of pairwise disjoint boxes, then 
$$
\sum_i \|f\|^2_{L^1(Q'_i)} \le \|f\|_{L^1(\cup Q'_i)} \sum_i\|f\|_{L^1(Q'_i)} = \|f\|^2_{L^1(\cup Q'_i)}.
$$
We cover the box $B_R$ by sub-boxes $(Q_i)$ of size~$r$. We can do so in such a way that no point of $B_R$ is covered by more than $2^d$ sub-boxes, so that
$$
\sum_i \|f\|^2_{L^1(Q_i)} \lesssim \|f\|^2_{L^1_R}.
$$
Combining this and \eqref{e:step13} into \eqref{e:step12} thus yields
$$
\|f\|^2_{L^2_R} \le  \sum_i \|f\|^2_{L^2(Q_i)} \lesssim |Q|^{2/d-2/q} \, |B_R|^{2/q} \, \M^2 \, \|w \na f \|^2_{L^2_R} + |Q|^{-1} \, \|f\|^2_{L^1_R},
$$
where we now simply use $|Q|$ to denote the cardinality of a box of size $r$, without reference to a specific box. Using \eqref{e:step10} (and $|Q|\ge r$), we obtain
$$
\|f\|^2_{L^2} \lesssim |Q|^{2/d-2/q} \, |B_R|^{2/q} \, \M^2 \, \|w \na f \|^2_{L^2} + r^{-d} \, \|f\|^2_{L^1} + R^{-p} \||x|_*^{p/2} f\|_{L^2}^2.
$$
By changing the implicit multiplicative constant in this inequality, we can ensure that
\begin{equation}
\label{e:to-be-optmi}
\|f\|_{L^2}  \lesssim r^{1-\frac{d}{q}} \, R^{\frac{d}{q}} \, \M \, \|w \na f \|_{L^2} + r^{-d/2} \, \|f\|_{L^1} + R^{-p/2} \, \||x|_*^{p/2} f\|_{L^2}
\end{equation}
holds uniformly over $R \ge r \ge 1$. In view of \eqref{e:step11.5}, we can extend this to all $R \ge 1$, $r \ge 1$. Since we clearly have $\|f\|_{L^2} \le \|f\|_{L^1}$ and $\|f\|_{L^2} \le \||x|_*^{p/2} f\|_{L^2}$, the inequality in fact holds for every $R > 0$ and $r > 0$. The conclusion will now follow by optimizing over $R$ and $r$. We summarize this optimization into the following lemma, whose proof is postponed to the end of this section.
\begin{lem}
\label{l:opt} 
Let $a,a',b,c > 0$, and define $\sigma = ab+a'c+bc$. For every $A, B,D \ge 0$,
$$
\inf_{r,R > 0} \Ll(R^a \, r^{a'} \, A + r^{-b} \, B + R^{-c} \, D\Rr) \le 3 A^{bc/\sigma} B^{a'c/\sigma} D^{ab/\sigma}.
$$

\end{lem}
In our context, we let
$$
a = \frac{d}{q}, \qquad a' = 1-\frac{d}{q}, \qquad b = \frac{d}{2}, \qquad c = \frac{p}{2},
$$
then $\sigma = ab+a'c+bc$, and finally 
$$
\alpha = \frac{bc}{\sigma}, \qquad \beta = \frac{a'c}{\sigma}, \qquad \gamma = \frac{ab}{\sigma}.
$$
Applying Lemma~\ref{l:opt} to \eqref{e:to-be-optmi} leads to \eqref{e:nash}. 
In order to check that $\alpha, \beta$ and $\gamma$ can be rewritten as in \eqref{e:def:alpha} for some $\theta$, we observe that
$$
(d+2) a'c + (p+2) ab = 2 a'c + 2ab + d\Ll(1-\frac{d}{q}\Rr)\frac{p}{2} + p\frac{d}{q}\frac{d}{2} = 2\sigma,
$$
so that
\begin{equation}
\label{e:bary}
\frac{d+2}{2}\, \beta + \frac{p+2}{2} \, \gamma = 1.
\end{equation}
This, together with the fact that $\alpha + \beta + \gamma = 1$, ensures a representation of the form \eqref{e:def:alpha} for some $\theta$. The value of $\theta$ can be recovered from 
\begin{multline*}
\frac{1}{\theta} = \frac{2}{p+2}\frac{1}{\gamma} = \frac{2}{p+2} \frac{\sigma}{ab} = \frac{2}{p+2}\Ll[1 + \frac{2q}{d^2}\Ll( \Ll(1-\frac{d}{q}\Rr)\frac{p}{2} + \frac{dp}{4} \Rr) \Rr] \\
= \frac{2}{p+2}\Ll[1 + \frac{p}{2} + \Ll(\frac{p}{d} + \frac{p}{2}\Rr)\Ll(\frac{q}{d} - 1 \Rr) \Rr],
\end{multline*}
and we recognize that $\theta = \theta_c$ as defined by \eqref{e:def:theta}.
\end{proof}

\begin{proof}[Proof of Lemma~\ref{l:opt}]
It suffices to consider the case when none of $A$, $B$ and $D$ are zero. We simply choose $r$ and $R$ so that the contributions of the three terms in the sum to be minimized are equal, that is,
\begin{equation}
\label{e:opt}
R^a \, r^{a'} \, A = r^{-b} \, B = R^{-c} \, D.
\end{equation}
The second equality gives
$$
R = \Ll(\frac{D}{B}\Rr)^{1/c} r^{b/c},
$$
while by the first equality,
$$
\Ll(\frac{D}{B}\Rr)^{a/c} r^{ab/c+a'+b} = \frac{B}{A}.
$$
This leads to 
$$
r^{\sigma} = A^{-c} \, B^{c+a} \, D^{-a},
$$
and thus
$$
r^{-b} \, B = A^{bc/\sigma} \, B^{1-\frac{ab+bc}{\sigma}} \, D^{ab/\sigma},
$$
and the lemma follows by \eqref{e:opt}.
\end{proof}

%
%
%
%
%
%
%
%

%
%
%
%
%
%
%
%
\section{Static environments}
\label{s:static}

We now show how to use the anchored Nash inequality derived in the previous section to deduce heat kernel bounds for static, degenerate environments.

\begin{defi}
\label{def:moderate}
Let $w : \B \to \R_+$. For a static environment $a : \B \to [0,1]$, we define $p_t(x,\cdot)$ to be the unique bounded function such that $p_0(x,y)  = \1_{x=y}$ and
$$
\partial_t p_t(x,y)  = \sum_{z \sim y} a(\{y,z\}) \Ll(p_t(x,z) - p_t(x,y)\Rr).
$$ 
In probabilistic terms, $p_t(x,y)$ is the probability for the random walk in the environment $a$ started at $x$ to be at $y$ at time $t$.
We write $u_t = p_t(0,\cdot)$ and
$$
\mcl{D}_t = \|\sqrt{a} \, \na u_t\|_2^2.
$$
We say that the static environment $a$ is $w$\emph{-moderate} if for every $t\ge0$,
\begin{equation}
\label{e:cond-stat1}
\|w \, \na u_t\|_2^2 \le \mcl{D}_t.
\end{equation}
\end{defi}

\begin{thm}[Upper bound for static environments]
For every $q > d$ and every $\alpha,\beta,\gamma \in (0,1)$ as in Theorem~\ref{t:nash}, there exists $C < \infty$ such that if the static environment~$a$ is $w$-moderate, then for every $t \ge 1$,
$$
p_t(0,0) \le C \  \frac{\M^{2\alpha/\beta}}{t^{d/2}},
$$
where we recall that $\M = \M(w)$ was defined in \eqref{e:def:Mq}.
\label{t:stat1}
\end{thm}

Before proving the theorem, we introduce some notation, and then prove an important preliminary result.
For $f : \Z^d \to \R$ and $x \in \Z^d$, we define $\na_i f(x)=f(x+\e_i) - f(x)$. We observe that the following discrete Leibniz rule holds, for $f,g : \Z^d \to \R$:
$$
\na_i(fg) = (\na_i f) g + f(\cdot + \e_i) \na_i g.
$$

We write $a_i(x) = a(x,x+\e_i)$.

\begin{prop}
\label{p:wrap}
There exists $C < \infty$ such that 
\begin{equation}
\label{e:wrap}
\frac{d}{dt} \, \||x|_*^{p/2} u_t\|_2^2 \le C \, \| |x|_*^{p/2} u_t\|_2^{2(p-2)/p} \, \|u_t\|_2^{4/p}.
\end{equation}
\end{prop}
\begin{proof}
The proof is similar to that of \cite[(81)]{gno}, with some extra care required by the fact that we do not assume uniform ellipticity here. We write $\rho(x) = |x|_*$, and observe that
\begin{eqnarray*}
\frac 1 2 \frac{d}{dt} \, \||x|_*^{p/2} u_t\|_2^2 & = & \frac 1 2 \frac{d}{dt} \sum_{x \in \Z^d} \rho^p u_t^2(x) \\
& = & -\sum_{e \in \B} \na(\rho^p u_t) a \na u_t (e) \\
& = & -\sum_{\substack{x \in \Z^d \\ 1 \le i \le d}}  \na_i(\rho^p u_t) a_i \na_i u_t (x).
\end{eqnarray*}
By the discrete Leibniz rule, we have 
$$
\na_i(\rho^p u_t) = \na_i(\rho^p) u_t + \rho^p(\cdot + \e_i) \na_i u_t,
$$
and moreover, $\na_i(\rho^p)(x) \lesssim \rho^{p-1}(x)$.
As a consequence, the \lhs\ of \eqref{e:wrap} is bounded by 
$$
\sum_{\substack{x \in \Z^d \\ 1 \le i \le d}} \Ll[-\rho^p(\cdot + \e_i)a_i |\na_iu_t|^2 + C \rho^{p-1} u_t a_i \na_i u_t \Rr] (x)
$$
for some constant $C$. Since $a \le 1$, we have $a \ge a^2$, so that by the Cauchy-Schwarz inequality,
$$
\sum_{\substack{x \in \Z^d \\ 1 \le i \le d}} \rho^{p-1} u_t a_i \na_i u_t (x) \le \Bigg(\sum_{\substack{x \in \Z^d \\ 1 \le i \le d}} \rho^p a_i |\na_i u_t|^2 (x)\Bigg)^{1/2} \Bigg(\sum_{\substack{x \in \Z^d \\ 1 \le i \le d}} \rho^{p-2} |u_t|^2 (x)\Bigg)^{1/2}.
$$
Since $\rho(x) \lesssim \rho(x+\e_i)$, the Young inequality implies that the \lhs\ of \eqref{e:wrap} is bounded by a constant times
$$
\sum_{x \in \Z^d} \rho^{p-2} |u_t|^2 (x) \stackrel{\text{(H\"older)}}{\le} \Ll(\sum_{x \in \Z^d} \rho^{p} |u_t|^2 (x) \Rr)^{\frac{p-2}{p}} \Ll( \sum_{x \in \Z^d} |u_t|^2 (x) \Rr)^{\frac2p},
$$
and this completes the proof.
\end{proof}
\begin{proof}[Proof of Theorem~\ref{t:stat1}]
Let $\mcl{E}_t = \|u_t\|_2^2$ and $\mcl{N}_t = \||x|_*^{p/2} u_t\|_2^2$.
We write ${\mcl{E}}'_t$ for the time derivative of $\mcl{E}$ at time $t$, and similarly for other quantities. Note that
\begin{equation}
\label{e:rel1}
{\mcl{E}}'_t = -2 \mcl{D}_t,
\end{equation}
while by Proposition~\ref{p:wrap}, 
\begin{equation}
\label{e:rel2}
\mcl{N}'_t \lesssim\mcl{N}_t^{(p-2)/p} \, \mcl{E}_t^{2/p}.
\end{equation}
Let us define 
$$
\Lambda_t = 1 \vee \sup_{s \le t} s^{d/2} \, \mcl{E}_s.
$$
Integrating \eqref{e:rel2} (and since $\mcl{N}_0 = 1$), we obtain, for every $t \ge 1/2$,
\begin{equation}
\label{e:rel2bis}
\mcl{N}_t \lesssim \Lambda_t \, t^{(p-d)/2}. 
\end{equation}
For $\alpha,\beta,\gamma \in (0,1)$ as in Theorem~\ref{t:nash}, we have
$$
\mcl{E}_t \lesssim \Ll(\M^2\, \|w \na u_t\|_2^2\Rr)^{\alpha} \ \mcl{N}_t^\gamma,
$$
where we used the fact that $\|u_t\|_1 = 1$. In view of \eqref{e:cond-stat1} and \eqref{e:rel1}, it follows that
$$
\mcl{E}_t \lesssim (-\M^2\, {\mcl{E}}'_t)^{\alpha} \ \mcl{N}_t^\gamma,
$$
and by \eqref{e:rel2bis}, for every $t \ge 1/2$,
$$
 (-\M^2\, \mcl{E}'_t)^{\alpha} \gtrsim \mcl{E}_t \ \Lambda_t^{-\gamma} \, t^{-\gamma(p-d)/2}.
$$
Integrating this relation (and recalling that $\Lambda_t$ is increasing), we get
\begin{equation}
\label{e:stt}
\mcl{E}_t^{1-\frac{1}{\alpha}} \gtrsim \M^{-2} \, \Lambda_t^{-\frac{\gamma}{\alpha}} \, t^{1-\frac{\gamma(p-d)}{2\alpha}},
\end{equation}
provided that $1-\frac{\gamma(p-d)}{2\alpha} > 0$. In order to check this and to simplify this expression, we recall from \eqref{e:bary} that
$$
\frac{d}{2} \, \beta + \frac{p}{2} \, \gamma = 1-\beta - \gamma = \alpha,
$$
so that
\begin{equation}
\label{e:barbar}
\alpha-\frac{\gamma(p-d)}{2} = \frac{d}{2} \, \beta + \frac{d}{2} \, \gamma = \frac{d}{2}(1-\alpha).
\end{equation}
Since $\alpha < 1$, the inequality \eqref{e:stt} is justified, and can be rewritten as
$$
\mcl{E}_t \lesssim \M^{\frac{2\alpha}{1-\alpha}} \, \Lambda_t^{\frac{\gamma}{1-\alpha}} \, t^{-\frac{d}{2}}.
$$
Since $\Lambda_t$ is increasing, we get that for every $s \le t$, 
$$
s^{d/2} \mcl{E}_s \lesssim \M^{\frac{2\alpha}{1-\alpha}} \, \Lambda_t^{\frac{\gamma}{1-\alpha}},
$$
that is,
$$
\Lambda_t \lesssim \M^{\frac{2\alpha}{1-\alpha}} \, \Lambda_t^{\frac{\gamma}{1-\alpha}}.
$$
Since $\frac{\gamma}{1-\alpha} = \frac{\gamma}{\beta + \gamma} < 1$, this shows that
$$
\Lambda_t \lesssim \M^{2\alpha/\beta},
$$
and in particular,
\begin{equation}
\label{e:integrated}
\mcl{E}_t \lesssim \frac{\M^{2\alpha/\beta}}{t^{d/2}}.
\end{equation}
By the semi-group property,
$$
p_t(0,0) = \sum_{x \in \Z^d} p_{t/2}(0,x) p_{t/2}(x,0).
$$
The symmetry of $p_t(\cdot,\cdot)$ thus leads to
$
p_t(0,0) \le \mcl{E}_{t/2},
$
and in view of \eqref{e:integrated}, this concludes the proof.
\end{proof}
As a by-product of the above proof, we also obtain the following result.
\begin{prop}[Off-diagonal decay]
\label{p:off-diag}
For every $p > d$, $q > d$ and $\alpha,\beta,\gamma \in (0,1)$ as in Theorem~\ref{t:nash}, there exists $C < \infty$ such that if the static environment $a$ is $w$-moderate, then for every $t \ge 1$,
$$
\||x|_*^{p/2} \, u_t\|_2^2 \le C \,\M^{2\alpha/\beta} \, t^{(p-d)/2}. 
$$
\end{prop}
\begin{proof}
This follows directly from \eqref{e:rel2bis} and \eqref{e:integrated}.
\end{proof}

\begin{prop}[Moment condition]
\label{p:moments}
Let $q > d$ and $\alpha, \beta, \gamma$ be as in Theorem~\ref{t:nash}. Assume that under the probability measure $\P$, the family of random variables $a = (a(e))_{e \in \B}$ is stationary with respect to translations, takes values in $(0,1]$, and satisfies
$$
\E\Ll[a(e)^{-q/2}\Rr] < \infty.
$$
There exists a random variable $\mcl X \ge 0$ such that
$$
\forall r < \frac{q\beta}{2\alpha},\ \E\Ll[\mcl X^r\Rr] < \infty
$$
and
\begin{equation}
\label{e:decay}
p_t(0,0) \le \frac{\mcl X}{t^{d/2}}.
\end{equation}
\end{prop}
\begin{proof}
Clearly, the environment $a$ is $\sqrt{a}$-moderate. Hence, it suffices to check that for every $r < 1$, we have $\E[\M^{qr}(\sqrt{a})] < \infty$. This follows from Proposition~\ref{p:weak}.
\end{proof}
\begin{rem}
\label{r:perco}
The percolation case corresponds to assuming that $(a(e))_{e \in \B}$ are i.i.d.\ Bernoulli random variables. Obtaining heat-kernel upper bounds for super-critical percolation would require more work, which we will not pursue here (see \cite{matrem, bar} for previous work on this). In \cite{lno}, the following simpler situation is considered: all edges~$e$ pointing in a given direction satisfy $a(e) = 1$, while $(a(e))$ are i.i.d.\ Bernoulli for the other edges. In this case, one can show that the environment is $w$-moderate for some $w$ such that all moments of $w(e)^{-1}$ are finite (see \cite[Lemma~4]{lno}). We thus obtain \eqref{e:decay} for a random variable $\mcl X$ with finite moments of every order.
\end{rem}
\section{Dynamic environments}
\label{s:dynamic}

We now turn our attention to \emph{dynamic} degenerate environments.

\begin{defi}
\label{def:dyn-mod}
For every $e \in \B$, we give ourselves a piecewise continuous function
$$
\Ll\{
\begin{array}{lll}
\R & \to & [0,1] \\
t & \mapsto & a_t(e),
\end{array}
\Rr.
$$
and we call $a = (a_t(e))_{e \in \B, t \in \R}$ a \emph{dynamic environment}. (It is convenient, although inessential, to consider environments defined over the whole time line.) For every $x \in \Z^d$ and $s \in \R$, we define $(p_{s,t}(x,y))_{y \in \Z^d, t \ge s}$ as the unique bounded function such that $p_{s,s}(x,y) = \1_{x = y}$ and, for every $t \ge s$,
$$
\partial_t p_{s,t}(x,y) = \sum_{z \sim y} a_t(\{y,z\}) \Ll(p_{s,t}(x,z) - p_{s,t}(x,y)\Rr).
$$
In probabilistic terms, $p_{s,t}(x,y)$ is the probability that the random walk evolving in the dynamic environment $a$ and started at time $s$ and position $x$ reaches $y$ at time~$t$.
We write $u_t = p_{0,t}(0,\cdot)$ and
$$
\mcl{E}_t = \|u_t\|_2^2, \qquad \mcl{D}_t = \|\sqrt{a_t} \, \na u_t\|_2^2, \qquad \ov{\mcl{D}}_t = \int_{t}^{+\infty} K_{s-t} \, \mcl{D}_s \, \d s.
$$
Let $K : \R_+ \to \R_+$ and, for each $e \in \B$, let $t \in \R_+ \mapsto w_t(e) \in [0,1]$ be a measurable function. We say that the dynamic environment $a$ is $(w,K)$\emph{-moderate} if for every $t \ge 0$,
\begin{equation}
\label{e:hyp:dynamic}
\|w_t \, \na u_t\|_2^2 \le \ov{\mcl{D}}_t.
\end{equation}
\end{defi}
\begin{thm}[Energy upper bound for dynamic environments]
\label{t:dynamic}
Let $q > d$ and $\alpha, \beta, \gamma \in (0,1)$ be as in Theorem~\ref{t:nash}. There exists $C < \infty$ such that if the dynamic environment $a$ is $(w,K)$-moderate, then 
\begin{equation}
\label{e:step1}
\mcl{E}_t \le C \  \mcl{C}(K) \ \frac{\MM^{2\alpha/\beta}(w)}{t^{d/2}},
\end{equation}
where 
\begin{equation}
\label{e:def:max-space-time}
\MM^{-2}(w) = \inf_{t \ge 1} \frac{1}{t} \int_0^t \M^{-2}(w_s) \, \d s
\end{equation}
and
\begin{equation}
\label{e:defCK}
\mcl{C}(K) =1 \vee \frac{\|K\|_1^{\alpha/\beta}}{\|K\|^{(1-\alpha)/\beta}_{L^1([0,1])}}.
\end{equation}
\end{thm}
\begin{proof}
We define
$$
\ov{\mcl{E}}_t = \int_{t}^{+\infty} K_{s-t} \, \mcl{E}_s \, \d s, \qquad \mcl{N}_t = \||x|_*^{p/2} u_t\|_2^2.
$$
Proposition~\ref{p:wrap} remains valid for dynamic random environments, and gives us
\begin{equation}
\label{e:diffN}
\mcl{N}'_t \lesssim\mcl{N}_t^{(p-2)/p} \, \mcl{E}_t^{2/p}.
\end{equation}
Defining
$
\Lambda_t = 1 \vee \sup_{s \le t} s^{d/2} \, \mcl{E}_s,
$
we obtain as for the static case that for every $t \ge 1$,
\begin{equation}
\label{e:rel2ter}
\mcl{N}_t \lesssim \Lambda_t \, t^{(p-d)/2}. 
\end{equation}
By Theorem~\ref{t:nash}, 
$$
\mcl{E}_t \lesssim \Ll(\M^2(w_t)\, \|w_t \na u_t\|_2^2\Rr)^{\alpha} \ \mcl{N}_t^\gamma,
$$
which by assumption~\eqref{e:hyp:dynamic} and \eqref{e:rel2ter} leads to
$$
\mcl{E}_t \lesssim \Ll(\M^2(w_t)\,\ov{\mcl{D}}_t\Rr)^{\alpha} \ \Lambda^\gamma_t \, t^{\gamma(p-d)/2}.
$$
Since $\mcl{E}_t$ is decreasing, we have $\ov{\mcl{E}}_t \lesssim \|K\|_1\,  \mcl{E}_t$, and moreover, $\ov{\mcl{E}}'_t = -2\ov{\mcl{D}}_t$. As a consequence,
$$
\ov{\mcl{E}}_t \lesssim \|K\|_1\,\Ll(-\M^2(w_t)\,\ov{\mcl{E}}'_t\Rr)^{\alpha} \ \Lambda^\gamma_t \, t^{\gamma(p-d)/2},
$$
which can be rewritten as
\begin{equation}
\label{e:diffeq}
-\ov{\mcl{E}}'_t \ \ov{\mcl{E}}_t^{-\frac1\alpha} \gtrsim \|K\|_1^{-\frac1\alpha} \, \M^{-2}(w_t)\ \Lambda^{-\frac{\gamma}{\alpha}}_t \, t^{-\frac{\gamma(p-d)}{2\alpha}}.
\end{equation}
Since $\Gamma_t$ is increasing and $p > d$, we obtain
$$
\ov{\mcl{E}}_t^{1-\frac1\alpha} \gtrsim \|K\|_1^{-\frac1\alpha} \, \Lambda^{-\frac{\gamma}{\alpha}}_t \, t^{-\frac{\gamma(p-d)}{2\alpha}} \int_0^t \M^{-2}(w_s) \, \d s.
$$
Recalling \eqref{e:barbar} and the definition of $\MM = \MM(w)$ in \eqref{e:def:max-space-time}, we obtain that for every $t \ge 1$,
$$
t^{d/2} \, \ov{\mcl{E}}_t \lesssim \|K\|_1^{\frac{\alpha}{1-\alpha}} \MM^{\frac{2\alpha}{1-\alpha}} \, \Lambda_t^{\frac{\gamma}{1-\alpha}}.
$$
We now observe that, since $\mcl{E}_t$ is decreasing,
$$
\ov{\mcl{E}}_t \ge \mcl{E}_{2t} \int_0^{t} K_s \, \d s,
$$
and in particular, for $t \ge 1$, $\ov{\mcl{E}}_t \ge \|K\|_{L^1([0,1])}\,  \mcl{E}_{2t}$. Since moreover, $\Lambda_2 \le 2^{d/2}$, we get
$$
\Lambda_t \lesssim 1+\frac{\|K\|_1^{\frac{\alpha}{1-\alpha}}}{\|K\|_{L^1([0,1])}} \MM^{\frac{2\alpha}{1-\alpha}} \, \Lambda_t^{\frac{\gamma}{1-\alpha}},
$$
and this proves the theorem since $\mcl{C}(K) \ge 1$ and $\MM \ge 1$ (recall that $w \le 1$).
\end{proof}

As before, the proof above reveals more detailed information about off-diagonal decay of the heat kernel.
\begin{prop}[Off-diagonal decay]
\label{p:dyn-off-diag}
Let $p > d$, $q > d$ and $\alpha, \beta, \gamma \in (0,1)$ be as in Theorem~\ref{t:nash}. There exists $C < \infty$ such that if the dynamic environment $a$ is $(w,K)$-moderate, then for every $t \ge 1$,
$$
\||x|_*^{p/2} u_t\|_2^2 \le C \  \mcl{C}(K) \ \MM^{2\alpha/\beta}(w) \ t^{(p-d)/2}.
$$
\end{prop}
\begin{proof}
This follows from \eqref{e:rel2ter} and \eqref{e:step1}.
\end{proof}

For a dynamic environment $a$, we let $a\t$ be the dynamic environment defined, for all $s \in \R$, by $a\t_s = a_{t-s}$. 

\begin{cor}[Heat kernel upper bound for dynamic environments]
\label{c:dyn}
In the setting of Theorem~\ref{t:dynamic}, there exists $C < \infty$ such that if $a$ is $(w,K)$-moderate and $a\t$ is $(w\t,K\t)$-moderate, then
\begin{equation}
\label{e:c:dyn}
p_{0,t}(0,0) \le C \, \sqrt{\mcl{C}(K)\mcl{C}(K\t)} \, \frac{\Ll( \MM(w)\MM(w\t) \Rr)^{\alpha/\beta}}{t^{d/2}}
\end{equation}
\end{cor}
Corollary~\ref{c:dyn} will be obtained from the following lemma.
\begin{lem}[Space-time reversal]
\label{l:reverse}
For every $u \le v \in \R$, $t \in \R$ and $x,y \in \Z^d$,
$$
p_{u,v}(x,y) = p\t_{t-v,t-u}(y,x).
$$
\end{lem}
\begin{proof}
For $f : \Z^d \to \R$ of compact support, we write
$$
f_{u,v}(x) = \sum_{y \in \Z^d} p_{u,v}(x,y)f(y),
$$
$$
f\t_{u,v}(x) = \sum_{y \in \Z^d} p\t_{u,v}(x,y)f(y).
$$
The lemma is equivalent to the claim that for every $f$, $g$ of compact support,
$$
(f_{u,v} \, , \,  g) = (f \, , \, g\t_{t-v,t-u}),
$$
where $(\cdot,\cdot)$ denotes the scalar product on $L^2(\Z^d)$. That this relation is correct is clear if $t\mapsto a_t$ is piecewise constant. Otherwise, one can proceed via finite-volume approximations.
\end{proof}
\begin{proof}[Proof of Corollary~\ref{c:dyn}]
From the lemma,
$$
p_{t/2,t}(x,0) = p\t_{0,t/2}(0,x).
$$
Hence, a Cauchy-Schwarz inequality gives
$$
p_{0,t}(0,0) = \sum_{x \in \Z^d} p_{0,t/2}(0,x) \, p_{t/2,t}(x,0) \le \sqrt{\mcl{E}_{t/2} \, \mcl{E}\t_{t/2}}
$$
(with obvious notation), and the result thus follows from Theorem~\ref{t:dynamic}.
\end{proof}
We conclude this section by providing a method to build $(w,K)$ out of a dynamic environment $a$, so that $a$ is $(w,K)$-moderate.
\begin{prop}[A criterion for moderation]
\label{p:crit}
There exists $c < \infty$ such that the following holds. Let $k : \R_+ \to \R_+$ be such that $\int (1+t^2) k_t \, \d t < \infty$, let
$$
K_t = k_t + \int_t^{+\infty} s \, k_s \, \d s,
$$
and for a dynamic environment $a$, let
\begin{equation}
\label{e:def:w}
w_t^2 = \int_t^{+\infty} k_{s-t} \, a_s \, \d s.
\end{equation}
Then the dynamic environment $a$ is $(w,cK)$-moderate. 
\end{prop}
\begin{rem}
Recall that in order to get non-trivial bounds from Theorem~\ref{t:dynamic}, we still need to make sure that $w^{-q}$ is integrable for some $q > d$ (the fact that $K \in L^1$ is automatic from the assumption on $k$). This can be obtained if one knows that it is unlikely for $a_s$ to remain close to $0$ during a long time interval.
\end{rem}
\begin{proof}
The condition $\int (1+t^2) k_t \, \d t < \infty$ ensures that $K$ is integrable. We need to estimate
$$
\|w_t \, \na u_t\|_2^2 = \sum_{e \in \B} w_t^2 (\na u_t)^2(e) = \int_t^{+\infty} k_{s-t} a_s (\na u_t)^2(e) \, \d s.
$$
This quantity is bounded by
$$
\int_t^{+\infty} k_{s-t} \underbrace{\sum_{e \in \B} a_s (\na u_s)^2(e)}_{=\mcl{D}_s} \, \d s +  \int_t^{+\infty} k_{s-t}  \sum_{e \in \B}a_s (\na u_s - \na u_t)^2(e) \, \d s.
$$
Up to a constant (and since $a_s \le 1$), the second sum above is bounded by
\begin{eqnarray*}
\sum_{x \in \Z^d} (u_s - u_t)^2(x) & = & \sum_{x \in \Z^d} \Ll(\int_t^s \na^* a_{s'} \na u_{s'}(x) \, \d s'\Rr)^2 \\
& \stackrel{\text{(Jensen)}}{\le} & (s-t) \sum_{x \in \Z^d} \int_t^s \Ll( \na^* a_{s'} \na u_{s'}(x) \Rr)^2 \, \d s' ,
\end{eqnarray*}
which, up to a constant, is bounded by
$$
(s-t)  \int_t^s \sum_{e \in \B} a_{s'} (\na u_{s'}(e))^2 \, \d s' = (s-t)\int_t^s \mcl{D}_{s'} \, \d s'.
$$
We have thus shown that
\begin{eqnarray*}
\|w_t \, \na u_t\|_2^2 & \lesssim & \int_t^{+\infty} k_{s-t} \, \mcl{D}_s \, \d s + \int_t^{+\infty} (s-t) k_{s-t} \int_t^s\mcl{D}_{s'} \, \d s' \, \d s \\
& \lesssim & \int_t^{+\infty} k_{s-t} \, \mcl{D}_s \, \d s + \int_t^{+\infty} \d s' \, \mcl{D}_{s'} \int_{s'}^{+\infty} (s-t) k_{s-t} \, \d s,
\end{eqnarray*}
which is the desired result.
\end{proof}

%
%
%
%
%
%
%
\section{The exclusion process as dynamic random environment}
\label{s:example}
As an example of the results of the previous section, we study a random walk evolving in an environment that is a local function of the symmetric simple exclusion process on $\Z^d$, $d \ge 2$. Recall that the infinitesimal generator $\mcl L$ of the symmetric simple exclusion process is given by
\begin{equation}
\label{e:def:gen}
\mcl L f(\eta) = \sum_{e \in \B} (f(\eta^e) - f(\eta)),
\end{equation}
where $\eta = (\eta_e)_{e \in \B} \in \{0,1\}^\B$, $f$ is a function that depends on a finite number of coordinates of $\eta$, and 
$$
(\eta^e)_b = 
\Ll|
\begin{array}{ll}
1- \eta_e & \text{if } b = e,\\
\eta_e & \text{otherwise}.
\end{array}
\Rr.
$$
This process is also called the Kawasaki dynamics. (Strictly speaking, the exclusion process is thought of as particles performing independent random walks, but with any jump of one particle onto another being suppressed; since the particles are indistinguishable here, this is equivalent to the process defined by \eqref{e:def:gen}.) We refer to \cite[Part~III]{lig} for a thorough study of this process.

For every $\rho \in [0,1]$, let $\mu_\rho = \bigotimes_{\Z^d} \textsf{Bernoulli}(\rho)$. The measures $\mu_\rho$ are the only extremal invariant (and in fact, reversible) measures of the process \cite[Corollary~III.1.11]{lig}. From now on, we fix $\rho < 1$. By stationarity, we can build the probability measure $\P$ such that under this measure, the process~$(\eta_t)_{t \in \R}$ (defined over the whole time line) evolves as an exclusion process, and $\eta_t$ is distributed according to $\mu_\rho$ for every $t \in \R$. 

For concreteness, we define our dynamic environment by
$$
a_t(e) = 
\Ll|
\begin{array}{ll}
1 & \text{if } \eta_t(\un e) = \eta_t(\ov e) = 0, \\
0 & \text{otherwise},
\end{array}
\Rr.
$$
where $\un e$ and $\ov e$ are the end-points of $e$. The specific form of $a$ is not important, as long as it is a function of the \emph{unoriented} edge $e$, and depends only locally on the configuration $\eta$. We let $(p_{s,t}(x,y))_{s \le t \in \R, x,y \in \Z^d}$ be as in Definition~\ref{def:dyn-mod}.
\begin{thm}[Heat kernel estimates]
\label{t:hk}
There exists a random variable $\mcl X$ with finite moments of every order and such that
\begin{equation}
\label{e:energy}
\sum_{x \in \Z^d} \Ll(p_{0,t}(0,x)\Rr)^2 \le \frac{\mcl X}{t^{d/2}}.
\end{equation}
Moreover, there exists a stationary process $(\mcl Y_t)_{t \in \R}$ satisfying
\begin{equation}
\label{e:point-dyn}
p_{0,t}(0,0) \le \frac{\mcl Y_t}{t^{d/2}},
\end{equation}
and such that for every $r > 0$ and $\eps > 0$,
\begin{equation}
\label{e:stat-mom}
\E \Ll[ \Ll( \mcl Y_t \Rr) ^r \Rr]  < \infty,
\end{equation}
\begin{equation}
\label{e:decay-mom}
\E \Ll[ \Ll( \sup_{t \ge 1}  \frac{\mcl Y_t}{t^\eps} \Rr) ^r \Rr]  < \infty.
\end{equation}
\end{thm}
\begin{rem}
We believe that $\sup_t t^{d/2} p_{0,t}(0,0)$ should have finite moments of every order, but proving this would require more work, so we do not pursue this question further here.
\end{rem}
We begin by showing that $a_t(e)$ is unlikely to remain equal to $0$ for a long period of time. 
\begin{lem}
\label{l:control}
For every $\kappa < 1$, there exists $C < \infty$ such that
$$
\P \Ll[ \int_0^t a_s(e) \, \d s \le 1 \Rr]  \le C \exp \Ll( -t^\kappa \Rr) .
$$
\end{lem}
\begin{proof}
Throughout the proof, the value of the constants $c > 0$ and $C < \infty$ may vary in each occurence. Let $\delta \in (1/d,1)$, and denote by $A_{s,t}$ the event 
$$
\text{there is at most one } x \in B_{\log^{\delta} t} \text{ s.t. }  \eta_s(x) = 0 .
$$
For every $s \in \R$ and $t \ge 2$,
\begin{equation}
\label{e:wholes}
1-\P \Ll[ A_{s,t}\Rr] \le C \exp \Ll(- c\log^{\delta d} t \Rr) .
\end{equation}
Hence,
$$
1-\P \Ll[ \forall k \text{ integer, } 0 \le k \le \frac{t}{\log^\delta t}-1, \ A_{k \log^\delta t,t} \text{ holds}\Rr] \le C\exp \Ll(- c\log^{\delta d} t \Rr).
$$
Let $e$ be an edge adjacent to the origin. Conditionally on $A_{0,t}$, we can estimate the probability that $a_s(e)$ becomes $1$ by considering a particular choice of moves that bring the two holes in $B_{\log^{\delta} t}$ onto the two endpoints of $e$, and then asking that they stay put for one unit of time. More precisely, for every $t \ge 2$,
$$
\P \Ll[ \int_0^{\log^\delta t} a_s(e) \ \d s \le 1 \ \big| \ A_{0,t} \Rr] \le 1- \exp\Ll(-C \log^\delta t\Rr)
$$
(the bound being very crude). By the Markov property,
\begin{multline*}
\P \Ll[ \int_0^t a_s(e) \, \d s \le 1 \Rr]  \le C\exp \Ll(- c\log^{\delta d} t \Rr) + \Ll[ 1-\exp\Ll(-C \log^\delta t\Rr) \Rr]^{\Ll\lfloor \frac{t}{\log^\delta t}\Rr\rfloor} \\
 \le C\exp \Ll(- c\log^{\delta d} t \Rr) + \exp \Ll(\Ll\lfloor \frac{t}{\log^\delta t}\Rr\rfloor \log \Ll[ 1-\exp\Ll(-C \log^\delta t\Rr) \Rr]\Rr). 
\end{multline*}
The first term in the sum above is controlled since $\delta > 1/d$. For the second term, the fact that $\delta < 1$ ensures that the logarithmic term decays to $0$ slower than any negative power of $t$, so the proof is complete.
\end{proof}
\begin{proof}[Proof of Theorem~\ref{t:hk}]
We let $k_t = (1+t)^{-4}$, and $K$, $w$ be defined as in Proposition~\ref{p:crit}, so that the dynamic environment $a$ is $(w,cK)$-moderate. For every $u \in (0,1)$,
\begin{align*}
\P[w_t(e) \le u] & = \P\Ll[\int_0^\infty k_{s} a_s(e) \, \d s \le \sqrt{u}\Rr]  \\
& \le  \P\Ll[(1+t)^{-4} \int_0^t  a_s(e) \, \d s \le \sqrt{u}\Rr],
\end{align*}
for arbitrary $t \ge 0$. We choose $t$ such that $(1+t)^{-4} = \sqrt{u}$, and apply Lemma~\ref{l:control} to obtain 
\begin{equation}
\label{e:control-tail}
\P[w_t(e) \le u] \le C \exp(-u^{-1/9}).
\end{equation}
In particular, $w_t(e)^{-1}$ has finite moments of every order. By Corollary~\ref{c:Lp}, the maximal function $\M(w_t)$ has finite moments of every order. The same property holds for $\MM(w)$ itself by Remark~\ref{r:inverses}. (Strictly speaking, Remark~\ref{r:inverses} applies only to processes indexed by a discrete time, but it is not difficult to check that this is sufficient for our purpose.) By Theorem~\ref{t:dynamic}, we thus obtain \eqref{e:energy} (the choice of $q \in (d,\infty)$ and of $\alpha, \beta, \gamma$ as in Theorem~\ref{t:nash} is arbitrary). Recall that we denote the time reversals around time $t$ of $a$ and $w$ by $a\t$ and $w\t$ respectively. Clearly, $a\t$ is $(w\t, cK)$-moderate, so Corollary~\ref{c:dyn} yields that
$$
p_{0,t}(0,0) \le \frac{\mcl Y_t}{t^{d/2}},
$$
with
$$
\mcl Y_t = C \Ll( \MM(w)\MM(w\t) \Rr)^{\alpha/\beta}
$$
for some constant $C < \infty$. Since $w\t$ and $w$ have the same law, every moment of $\MM(w\t)$ is finite, and \eqref{e:stat-mom} is proved. In order to obtain \eqref{e:decay-mom}, it suffices to consider a supremum over integer times, since for any $u \in [t-1,t]$, we have $\MM(w^{(u)}) \le 2 \MM(w\t)$. We then note that for any $r > 1/\eps$ and $y > 0$,
$$
\P \Ll[ \sup_{n \in \N \setminus \{0\}} n^{-\eps} \mcl Y_n \ge y\Rr] \le \sum_{n = 1}^{+\infty} \frac{\E\Ll[ \Ll(\mcl Y_n\Rr)^r\Rr]}{(yn^{\eps})^r} \le \frac{C}{y^r},
$$
so that \eqref{e:decay-mom} follows.
\end{proof}

\appendix

\section{Maximal inequalities}

In this appendix, we recall classical properties of maximal functions in a multi-dimensional setting. Let $(\Omega, \mcl F, \P)$ be a probability space,  and $(\theta_x)_{x \in \Z^d}$ be a measure-preserving action of $\Z^d$ on this space. For every measurable function $f : \Omega \to \R$, we define the maximal function
$$
\Ma f(\o) = \sup_{r \in \N} \frac{1}{|B_r|} \sum_{x \in B_r} f(\theta_x \o).
$$
The following result is \cite[Theorem~3.2]{akckre}.
\begin{prop}[Weak type (1,1) estimate]
\label{p:weak}
For every $f \in L^1(\Omega)$ and $\lambda > 0$,
$$
\P \Ll[ \Ma f \ge \lambda \Rr] \le \frac{3^d \, \|f\|_{L^1(\Omega)}}{\lambda}.
$$
\end{prop}
Obviously, the maximal function defines a bounded operator from $L^\infty$ to $L^\infty$. By the Marcinkiewicz interpolation theorem (see e.g.\ \cite[Appendix~D]{taylor}), we thus have the following.
\begin{cor}[$L^p$ estimate]
\label{c:Lp}
For every $p \in (1,\infty]$, there exists $C_p < \infty$ such that
$$
\|\Ma f\|_{L^p(\Omega)} \le C_p \|f\|_{L^p(\Omega)}.
$$
\end{cor}
\begin{rem}
\label{r:inverses}
If we let
$$
\mi g(\o) = \inf_{r \in \N} \frac{1}{|B_r|} \sum_{x \in B_r} g(\theta_x \o),
$$
then for every $p \in (1,\infty]$ and $g > 0$, we also have
$$
\|\Ll(\mi g\Rr)^{-1}\|_{L^p(\Omega)} \le C_p \|g^{-1}\|_{L^p(\Omega)},
$$
since by Jensen's inequality, $\Ll(\mi g\Rr)^{-1} \le \Ma(g^{-1})$.
\end{rem}

\bibliographystyle{abbrv}
\bibliography{Nash}

\end{document}